\documentclass[11pt]{amsart}
\usepackage{amsfonts}
\usepackage{xcolor}
\usepackage{amssymb}
\usepackage{enumerate}
\usepackage{comment}
\usepackage{mathtools}
\usepackage{tcolorbox}

\newtheorem{proposition}{\textbf{Proposition}}
\newtheorem{theorem}{\textbf{Theorem}}

\newtheorem{corollary}{\textbf{Corollary}}

\newtheorem{lemma}{\textbf{Lemma}}

\newtheorem*{questionC}{\textbf{Problem C}}
\newtheorem*{questionCinCp}{\textbf{Problem C in $\C_p$}}
\theoremstyle{definition}

\newtheorem{remark}{\textbf{Remark}}

\def\N {\mathbb{N}}
\def\Z {\mathbb{Z}}

\def\Q {\mathbb{Q}}

\def\QQ {\overline{\Q}}
\def\C {\mathbb{C}}

\theoremstyle{remark}

\numberwithin{equation}{section}

\begin{document}

	\title[On the Exceptional Sets of $p$-adic Transcendental Functions]{On the Exceptional Sets of $p$-adic Transcendental Analytic Functions}

	\author{Bruno De Paula Miranda}
	\address{Instituto Federal de Goi\'{a}s, Avenida Saia Velha, Km 6, BR-040, s/n, Parque Esplanada V, Valpara\'{i}so de Goi\'{a}s, GO 72876-601, Brazil}
	\email{bruno.miranda@ifg.edu.br}

	\author{JEAN LELIS}
	\address{Faculdade de Matemática/ICEN/UFPA, Belém - PA, Brazil.}
	\email{jeanlelis@ufpa.br}


	%    General info
	%11J61Approximation in non-Archimedean valuations
	%11S80Other analytic theory (analogues of beta and gamma functions, $p$-adic integration, etc.) 
	
	
\begin{abstract}
In this paper, we study the exceptional sets $S_f$ of $p$-adic transcendental analytic functions $f$ with rational and algebraic coefficients. We establish a necessary condition for a subset $S \subseteq \overline{\mathbb{Q}} \cap B(0, \rho)$ to be the exceptional set of a $p$-adic transcendental analytic function with rational coefficients, demonstrating that, in general, the answer to Mahler's Problem C over $\mathbb{C}_p$ is negative. However, we prove that if $S$ is closed under algebraic conjugation and contains 0, there exist uncountably many transcendental analytic functions $f \in \mathbb{Q}_{\rho}[[z]]$ such that $S_f = S$. Furthermore, if $\rho \geq 1$, $f$ can be taken in $\Z_{\rho}[[z]]$. Additionally, we demonstrate that any $S \subseteq \overline{\mathbb{Q}} \cap B(0, \rho)$ containing 0 can be the exceptional set of uncountably many transcendental analytic functions $f \in \overline{\mathbb{Q}}_{\rho}[[z]]$.
\end{abstract}

	\maketitle

		%}
	%%  SECTION 1

        \section{Introduction and Motivation}

  Many interesting classes of complex numbers which are transcendental over $\Q$ are the values of suitable \textit{transcendental} analytic functions. We recall that a transcendental analytic function is an analytic function $f:\Omega\subseteq \mathbb{C}\to \mathbb{C}$ such that the only polynomial $P(X,Y)$ in $\mathbb{C}[X,Y]$ that satisfies 
\[
P(z,f(z))=0\quad\mbox{for all}\quad z\in\Omega
\]
 is the null polynomial. For instance, the Hermite-Lindemann theorem implies that the transcendental function $f(z)=e^z$ takes transcendental values for all non-zero algebraic numbers. The Gelfond-Schneider theorem and Baker's theorem show that the transcendental functions $g(z)=2^z$ and $h(z)=e^{z\pi+1}$ take transcendental values for all non-rational algebraic numbers and all algebraic numbers, respectively. Assuming Schanuel's Conjecture, we have that the transcendental functions $r(z)=2^{2^{z}}$ and $t(z)=2^{2^{2^{z-1}}}$ take transcendental values for all non-integer and non-positive integer algebraic numbers, respectively.

The study of the arithmetic behavior of a transcendental analytic function in a complex variable began in 1886 with a letter from Weierstrass to Strauss, who proved the existence of such functions mapping $\mathbb{Q}$ into itself. Weierstrass also conjectured the existence of a transcendental entire function $f$ for which $f(\overline{\mathbb{Q}}) \subseteq \overline{\mathbb{Q}}$ (as usual, $\overline{\mathbb{Q}}$ denotes the field of algebraic numbers). Motivated by this kind of study, he defined the \textit{exceptional set} of an analytic function $f: \Omega \subseteq \mathbb{C} \to \mathbb{C}$ as
\[
S_f := \{\alpha \in \overline{\mathbb{Q}} \cap \Omega : f(\alpha) \in \overline{\mathbb{Q}}\}.
\]
Weierstrass' conjecture was settled in 1895 by St\"ackel \cite{stackel1895}.

In his book, Mahler \cite{Mahler3} introduces the problem of studying exceptional sets $S_f$ of transcendental analytic functions $f$ with coefficients in some prescribed set $A$. He denoted by $A_{\rho}[[z]]$ the set of all power series 
\[
f(z)=\sum_{n=0}^{\infty}a_nz^n,
\]
with coefficients in $A$ and radius of convergence $\rho$. After discussing a number of results and examples, Mahler raised some problems about the admissible exceptional sets for these analytic functions. One of them is: 
\begin{questionC}
Let $\rho\in (0,\infty]$ be a radius of convergence. Does there exist, for any choice of $S\subseteq \QQ\cap B(0,\rho)$,  a transcendental analytic function $f\in\mathbb{Q}_{\rho}[[z]]$ for which $S_f=S$?
\end{questionC}

Naturally, since $f$ has rational coefficients, in order to have $S = S_f$, we must have $0 \in S$. Furthermore, since $f(\overline{z}) = \overline{f(z)}$, where the bar denotes complex conjugation, $S$ must be closed under complex conjugation. In 2016, Marques and Ramirez \cite{ramirez2016} answered this question affirmatively for $\rho = \infty$ (i.e., for entire functions). Their result was improved by Marques and Moreira in \cite{gugu2018}, who provided an affirmative answer to Mahler's Problem {\sc C} for any $\rho \in (0, \infty]$.

In 2020, Marques and Moreira \cite{gugu20} solved Mahler's Problem C in $\mathbb{Z}\{z\}$. They proved that any subset of $\mathbb{Q}\cap B(0,1)$ (closed under complex conjugation and containing the element 0) is the exceptional set of a transcendental function analytic on $B(0,1)$ with integer coefficients. In this article, we will consider Mahler's Problem C in the $p$-adic context.

In the course of this work, $p$ denotes a fixed prime number, and $|\cdot|_p$ denotes the $p$-adic absolute value on the field $\Q$ of rational numbers, normalized such that $|p|_p=p^{-1}$.  We denote the unique extension of $|\cdot|_p$ to the field $\Q_p$ of $p$-adic numbers by the same notation $|\cdot|_p$. For completions sake, we recapitulate how to go from $\mathbb{Q}_p$ to $\mathbb{C}_p$. Let $K$ be a finite extension of $\mathbb{Q}_p$ of degree $n$. We know that in $K$ there is a unique absolute value extending the $p$-adic absolute value of $\mathbb{Q}_p$, which is given by 
\begin{equation}\label{AV}
|x| = \sqrt[n]{|N_{K/\mathbb{Q}_p}(x)|_p}
\end{equation}
for every $x \in K$, where $N_{K/\mathbb{Q}_p}(x)$ is the norm from $K$ to $\mathbb{Q}_p$ applied to $x$. Furthermore, we know that $K$ is complete with respect to this absolute value. In particular, in $K$ we have a $p$-adic valuation: for each $x \in K$, $\nu_p(x)$ is the only rational number in $\frac{1}{n}\mathbb{Z}$ satisfying 
$$|x| = p^{-\nu_p(x)}.$$ 

The image of the $p$-adic valuation $\nu_p$ of $K$ is called the \textit{value group} of $K$. We know that for each $n$ there are extensions $K$ of degree $n$ for which the value group is precisely $\frac{1}{n}\mathbb{Z}$ (they are called \textit{totally ramified extensions}). It is noteworthy that the value in \eqref{AV} does not depend on the extension $K$. That is, if the same $x$ is considered as an element of another extension $L/\mathbb{Q}_p$ of degree $m$, we have
$$\sqrt[m]{|N_{L/\mathbb{Q}_p}(x)|_p} = |x| = \sqrt[n]{|N_{K/\mathbb{Q}_p}(x)|_p}.$$ 
In light of these observations, we see that $\overline{\mathbb{Q}}_p/\mathbb{Q}_p$ is an extension of infinite degree, the $p$-adic absolute value of $\mathbb{Q}_p$ extends to $\overline{\mathbb{Q}}_p$ by means of \eqref{AV}, and the value group of $\overline{\mathbb{Q}}_p$ is precisely $\mathbb{Q}$. Now, remember that $\overline{\mathbb{Q}}_p$ is not complete; in fact, $\mathbb{C}_p$ is its completion. That is, each element $\alpha \in \mathbb{C}_p$ is the limit of a Cauchy sequence $(x_n)$ of elements of $\overline{\mathbb{Q}}_p$. In particular, because the absolute value in $\overline{\mathbb{Q}}_p$ is non-archimedean, there exists a natural number $N$ such that $|\alpha| = |x_n|$ for every $n > N$. We conclude that the absolute value on $\mathbb{C}_p$ extends the absolute value on $\mathbb{Q}_p$, $\mathbb{C}_p$ is complete, $\overline{\mathbb{Q}}_p$ is dense in $\mathbb{C}_p$, and the value group of $\mathbb{C}_p$ is precisely $\mathbb{Q}$. Since $\mathbb{Q}$ is dense in $\mathbb{Q}_p$, we have that $\overline{\mathbb{Q}}$ is dense in $\overline{\mathbb{Q}}_p$ and consequently in $\mathbb{C}_p$. Moreover, the value group of $\overline{\mathbb{Q}}$ is also $\mathbb{Q}$.

A \textit{$p$-adic transcendental analytic function} over $\mathbb{C}_p$ is an analytic function $f:\Omega\subseteq \mathbb{C}_p\to \mathbb{C}_p$ such that the only polynomial $P(X,Y)$ in $\mathbb{C}_p[X,Y]$ satisfying 
\[
P(z,f(z))=0 
\]
for all $z \in \Omega$, is the null polynomial.

Many results about transcendental numbers have been extended to $p$-adic transcendental numbers. For instance, Mahler in \cite{mahler1} proved a $p$-adic analogue of Lindemann's theorem: if $\alpha\in\mathbb{C}_p$ is algebraic over $\mathbb{Q}$ and $0<|\alpha|<p^{-1/(p-1)}$, then the value $\exp_p(\alpha)$ of the $p$-adic exponential function is transcendental. Here, we denote by $\exp_p(z)$ the function given by
\[
\exp_p(z)=\sum_{n=0}^{\infty}\frac{z^n}{n!},
\]
which converges for all $z$ in $\{z\in\mathbb{C}_p: |z|<p^{\frac{-1}{p-1}}\}$. Moreover, in \cite{mahler2} he proved a $p$-adic analogue of the Gelfond-Schneider theorem (Hilbert's 7th problem): if $\alpha,\beta\in\mathbb{C}_p$ are algebraic over $\mathbb{Q}$ satisfying $0<|\alpha-1|<p^{-1/(p-1)}$ and $0<|\beta-1|<p^{-1/(p-1)}$, then $\frac{\log_p(\alpha)}{\log_p(\beta)}$ is either rational or transcendental. 

Here, $\log_p(z)$ denotes the function 
\[
\log_p(1+z) = \sum_{n=1}^{\infty}(-1)^{n+1}\frac{z^n}{n},
\] 
and satisfies $\exp_p(\log_p(1+z))=1+z$. We recommend reading \cite{Bugeaud2018,Bugeaud2020,Bugeaud,DePaula,Lelis1,mahler1,mahler2,Ooto} and the references therein for more details on $p$-adic transcendental numbers and functions.

In this paper, we adopt a natural generalization of Weierstrass' definition of $S_f$ in the $p$-adic context: the exceptional set $S_f$ of a $p$-adic transcendental analytic  function $f:\Omega\subset\mathbb{C}_p\to\mathbb{C}_p$ is given by
\[
S_f:=\{\alpha\in\Omega\cap\QQ:f(\alpha)\in\QQ\}.
\]
Moreover, for a set $A\subseteq\C_p$ and a radius of convergence $\rho\in(0,\infty]$, we denote by $A_{\rho}[[z]]$ the set of all power series 
\[
f(z)=\sum_{n=0}^{\infty}a_nz^n
\]
with coefficients in $A$ and radius of convergence $\rho$. In this paper we bring to the $p$-adic context the problem introduced by Mahler about the study of exceptional sets $S_f$ of (now $p$-adic) transcendental analytic functions $f$ with coefficients in some prescribed set $A$. One of the problems we address is the $p$-adic version of Mahler's Problem $C$.
\begin{questionCinCp}
    Let $\rho\in(0,\infty]$ be a radius of convergence. Does there exist for any $S\subseteq\QQ\cap B(0,\rho)$ a $p$-adic transcendental analytic function $f\in\mathbb{Q}_{\rho}[[z]]$ for which $S_f=S$?
\end{questionCinCp}

Naturally, here we must also have $0 \in S$. However, in the $p$-adic case, we need a stronger restriction on the set $S$. We prove that if $f\in\mathbb{Q}_{\rho}[[z]]$ and $\alpha \in S_f$ is such that $[\mathbb{Q}_p(\alpha): \mathbb{Q}_p] > 1$, then $S_f$ must necessarily contain every $\alpha'\in\QQ$ which is $\Q_p$-conjugate of $\alpha$. We then add to Problem $C$ in $\C_p$ the extra hypothesis on closure under algebraic conjugation and  prove the following theorem.

\begin{theorem}\label{theo1}
     Let $\rho\in(0,\infty]$ and let $S$ be a subset of  $\ \overline{\mathbb{Q}}\cap B(0,\rho)$ closed under algebraic conjugation such that $0 \in S$. Then there exist uncountably many $p$-adic transcendental analytic functions $f\in\mathbb{Q}_{\rho}[[z]]$ such that $S_f=S$. Moreover, if $\rho\geq1$, then $f$ can be taken in $\Z_{\rho}[[z]]$.
\end{theorem}

We also show that if we loosen the restrictions on the coefficients of $f$ and look for functions in $\overline{\mathbb{Q}}_{\rho}[[z]]$, we can avoid the hypothesis on closure under algebraic conjugation. Specifically, we prove the following theorem.

\begin{theorem}\label{theo2}
Let $\rho\in(0,\infty]$ and  let $S$ be a subset of $\overline{\mathbb{Q}}\cap B(0,\rho)$ such that $0 \in S$. Then there exist uncountably many $p$-adic transcendental analytic functions $f\in\overline{\mathbb{Q}}_{\rho}[[z]]$ such that $S_f=S$.
\end{theorem}

To finish this introduction, it is appropriate to comment briefly on the quality of these results. We highlight that the proof of Theorem \ref{theo1} is inspired by Marques and Moreira \cite{gugu20}, where they work with transcendental analytic functions with integer coefficients and radius of convergence $\rho = 1$. However, in our case, besides the adjustments to the $p$-adic context, we introduce a new strategy for controlling the radius of convergence $\rho$ and also justify the additional restriction on closure under algebraic conjugation, which imposes some changes on the choice of the auxiliary interpolation functions. In the case of Theorem \ref{theo2}, the proof is similar to the one given by Marques, Moreira, and Ramirez in \cite{ramirez2016} and \cite{gugu20}, which uses a density argument. While they work with $\mathbb{Q}(i)$, which is dense in $\mathbb{C}$, we work with $\overline{\mathbb{Q}}$, which is dense in $\mathbb{C}_p$, and obtain the most general scenario possible regarding $S$.

\section{Analytic Function with Rational Coefficients}

Our first result on the exceptional set of $p$-adic transcendental analytic functions justifies a stronger restriction than in the complex case when the coefficients are rational. More precisely, we have the following proposition.

 \begin{proposition}\label{prop1}
 Let $f\in\Q_{\rho}[[z]]$ be a $p$-adic analytic function with radius of convergence  $\rho\in(0,\infty]$. If $\alpha \in S_f$ is such that $[\mathbb{Q}_p(\alpha): \mathbb{Q}_p] > 1$, then $S_f$ contains every $\alpha'\in\QQ$ which is $\Q_p$-conjugate of $\alpha$.
 \end{proposition}
 
 \begin{proof} Let $\rho\in(0,\infty]$ be the radius of convergence of the power series 
 \[
 f(z)=\sum_{n=0}^{\infty}a_nz^n
 \]
 with rational coefficients. Let $\alpha\in S_f$ be an algebraic number such that $[\Q_p(\alpha):\Q_p]>1$ and let $\alpha'\in\QQ$ be a $\Q_p$-conjugate of $\alpha$. By \eqref{AV}, we have that
 \[
 |\alpha|=|\alpha'|,
 \]
so $\alpha'\in B(0,\rho)$. 

 Now, suppose that $f(\alpha)=\beta$ and consider $P_{\beta}(z)\in\Q[z]$  the minimal polynomial of $\beta\in\QQ$. Thus, we define the power series 
 \[
 f_{\beta}(z)=P_{\beta}(f(z))=\sum_{n=0}^{\infty}b_nz^n.
 \]
 Since $f(z)\in\Q_{\rho}[[z]]$ and $P_{\beta}(z)\in\Q[z]$, we have that $f_{\beta}(z)\in\Q_{\rho}[[z]]$ and $f_{\beta}(\alpha)=0$. If $c$ is the absolute value of $\alpha$, then $0<c<\rho$ and $|b_n|c^n\to0$ as $n\to\infty$. Let $N$ be the integer number defined by the conditions
 \[
 |b_N|c^N=\max_{n}\{|b_n|c^n\}\quad\mbox{and}\quad |b_n|c^n<|b_N|c^N\quad\forall n>N\in\Z.
 \]
 By the $p$-adic Weierstrass Preparation Theorem \cite{Gouvea}, there exists a polynomial
 \[
 g_{\beta}(z)=c_0+c_1z+\cdots+c_Nz^N \in \Q_p[z]
 \]
 of degree $N$ and a power series
 \[
 h_{\beta}(z) = \sum_{n=0}^{\infty}d_nz^n \in \Q_p[[z]]
 \]
  satisfying
 \[
 f_{\beta}(z)=g_{\beta}(z)h_{\beta}(z),
 \]
 such that  $h_{\beta}(z)$ has no zeros in $\overline{B}(0,c)$. Therefore, $g_{\beta}(\alpha)=0$ and the minimal polynomial $P_{\alpha}(z)\in\Q_p[z]$ of $\alpha$ divides $g_{\beta}(z)$. We then conclude that, $ 0 = f_{\beta}(\alpha')=P_{\beta}(f(\alpha'))$ so $f(\alpha')$ is an algebraic conjugate of $\beta$ and the proof is finished.
 \end{proof}

\begin{remark}   
Note that the above result establishes a negative answer to Mahler's problem $C$  over $\mathbb{C}_p$. For instance, if the integer $1<b\leq p-1$ is not a quadratic residue modulo $p$, then $q(x)=x^2-b$ is irreducible over $\mathbb{Q}_p$. Suppose that $\alpha,\alpha' \in \overline{\mathbb{Q}}_p$ are the distinct roots of $q(x)$. By the above proposition, if $\alpha \in S_f$ for some $f\in\mathbb{Q}_{\rho}[[z]]$ with $\rho>|\alpha|$, then necessarily $\alpha'\in S_f$. Consequently, there does not exist a function $f\in\mathbb{Q}_{\rho}[[z]]$ such that $S_f=\{0,\alpha\}$.
\end{remark}

Now, in order to prove Theorem \ref{theo1}, we need the following auxiliary results. Our first lemma will be used to prove that there are only countably many power series with coefficients in a countable field.

\begin{lemma}
    Let $\mathbb{K}$ be a subfield of $\mathbb{C}_p$ and let $f \in \mathbb{K}_{\rho}[[z]]$ be an algebraic function of degree $n$ over $\mathbb{C}_p[z]$, for some $\rho\in(0,+\infty]$. Then $f$ is an algebraic function of degree $n$ over $\mathbb{K}[z]$.
\end{lemma}

\begin{proof}
If $f \in \mathbb{K}_{\rho}[[z]]$ is an algebraic function of degree $n$ over $\mathbb{C}_p[z]$ given by
\[
    f(z) = \sum_{k=0}^{\infty} a_k z^k,
\]
then there exists a non-zero polynomial $P(X,Y) = \sum c_{i,j} X^i Y^j \in \mathbb{C}_p[X,Y]$ of degree $n$ such that $P(z, f(z)) = 0$ for all $z \in B(0, \rho)$. Thus, we have
\[
    \sum_{0 \leq i+j \leq n} c_{i,j} z^i \left( \sum_{k=0}^{\infty} a_k z^k \right)^j = 0
\]
for all $z \in B(0, \rho)$. From here, expanding the powers, switching the order of summation, and rearranging the terms, we obtain
\[
    \sum_{k=0}^{\infty} A_k z^k = 0
\]
for all $z \in B(0, \rho)$, where each $A_k$ is a linear combination with coefficients in $\mathbb{K}$ of the $c_{i,j}$ for $0 \leq i+j \leq n$ . Since a power series is identically zero if and only if all its coefficients are zero, we have $A_k = 0$ for all $k \geq 0$. For each $k \geq 0$, denote by
\[
    A_k(x_{0,0}, \ldots, x_{0,n}) = 0
\]
the linear equation with coefficients in $\mathbb{K}$ obtained by replacing $c_{i,j}$ with the variable $x_{i,j}$. Of these infinite linear equations, at most $(n+1)(n+2)/2$ can be linearly independent, and they admit a non-trivial solution ($x_{i,j} = c_{i,j}$) over $\mathbb{C}_p$. Since the coefficients of each equation are in the field $\mathbb{K}$, the system also admits a non-trivial solution $x_{i,j} = \overline{c}_{i,j}$ such that $\overline{c}_{i,j} \in \mathbb{K}$ for all $0 \leq i+j \leq n$. Consequently,
\[
    \sum_{0 \leq i+j \leq n} \overline{c}_{i,j} z^i \left( \sum_{k=0}^{\infty} a_k z^k \right)^j = 0
\]
for all $z \in B(0, \rho)$. Taking $\overline{P}(X,Y) = \sum \overline{c}_{i,j} X^i Y^j \in \mathbb{K}[X,Y]$, we obtain that $\overline{P}(z,f(z))=0$ for all $z\in B(0,\rho)$. Since $\mathbb{K}$ is a subfield of $\C_p$ and $f$ is an algebraic function of degree $n$ over $\mathbb{C}_p[z]$, we have that the degree of $\overline{P}(X,Y)$ is greater than or equal to $n$. On the other hand, by construction, the degree of $\overline{P}(X,Y)$ is less than or equal to $n$. Hence, $f$ is an algebraic function of degree $n$ over $\mathbb{K}[z]$.
      \end{proof}

The lemma above is a $p$-adic version of a particular case of the result found in \cite{Mahler3}, which can be proved in its most general form involving a fixed number of derivatives of the function $f$ and for any field extension  $\mathbb{K}_0 \subseteq \mathbb{K}$ of characteristic 0. As a consequence of the above lemma, we obtain the following corollary.

\begin{corollary}\label{c1}
The set of power series with coefficients in a countable subfield of $\mathbb{C}_p$ with radius of convergence $\rho \in (0, +\infty]$ which are algebraic is countable.
\end{corollary}
    
Our next auxiliary result, proved by Bugeaud and Kekeç in \cite{Bugeaud}, is a \textit{Liouville's inequality} for $p$-adic numbers that are algebraic over $\mathbb{Q}$. To this end, we recall that for an integer polynomial $P(x)$, its height, denoted by $H(P)$, is the maximum of the usual absolute values of its coefficients, and its degree is denoted by $\deg(P)$. For $\alpha \in \overline{\mathbb{Q}}$, the height $H(\alpha)$ and the degree $\deg(\alpha)$ are the height and the degree of the minimal defining polynomial of $\alpha$ over $\mathbb{Z}$.

\begin{lemma}\label{Bugeaud}
     Let $\alpha,\beta\in\C_p$ be distinct algebraic numbers over $\Q$.
    \begin{enumerate}[(i)]
        \item If $\alpha$ and $\beta$ are algebraic conjugates such that $\deg(\alpha)=\deg(\beta)=n$ and $H(\alpha)=H(\beta)=H$, then
        \[
        |\alpha-\beta|\geq 2^{-3n/2}n^{-5n/2}H^{-2n}.
         \]
        \item If $\alpha$ and $\beta$ are not algebraic conjugates of degree $n$ and $m$, respectively, then
        \[
        |\alpha-\beta|\geq (n+1)^{-m}(m+1)^{-n}H(\alpha)^{-m}H(\beta)^{-n}\max\{1,|\alpha|\}\max\{1,|\beta|\}.
        \]
    \end{enumerate}
\end{lemma}

The lemma below is well-known and allows us to upper bound the height of the image of algebraic numbers by polynomials. Its proof can be found in \cite{gugu20}.

\begin{lemma}\label{HeightPoly}
Let $y_1, \ldots, y_k$ be algebraic numbers of degrees $m_1, \ldots, m_k$, respectively. Then the following inequalities hold:
\begin{enumerate}[(i)]
\item $H(y_1^n) \leq e^{O(n)}H(y_1)^n$ for all $n \geq 1$;
\item $H(y_1 y_2) \leq e^{O(1)}(H(y_1)H(y_2))^{m_1 m_2}$;
\item $H(y_1 + \cdots + y_k) \leq e^{O(k)}(H(y_1) \cdots H(y_k))^{m_1 \cdots m_k}$.
\end{enumerate}
\end{lemma}

\subsection*{Proof of Theorem \ref{theo1}} The demonstration consists in showing that there are uncountably many power series $h(z)$ and $g(z)$ such that $h(z)$ is a transcendental entire function with integer coefficients and $S_h = S$, and $g(z)$ has a radius of convergence $\rho$, rational coefficients, and $g(\overline{\mathbb{Q}} \cap B(0, \rho)) \subseteq \overline{\mathbb{Q}}$. In this way, since there exist only countably many algebraic power series in $\mathbb{Q}_{\rho}[[z]]$ (see Corollary \ref{c1}), we get uncountably many transcendental power series $f(z) = g(z) + h(z)$ such that $f \in \mathbb{Q}_{\rho}[[z]]$ and $S_f = S$. Furthermore, we will show that if $\rho \geq 1$, then $g(z)$ can be constructed with integer coefficients, and hence $f(z)$ too.

We first construct a transcendental entire function $h$ such that $S_h = S$. Let $S$ be a subset of $\overline{\mathbb{Q}}$ which is closed under algebraic conjugation and contains $0$. We define the set
\[
\mathcal{P}_S = \{P_1(z), P_2(z), \ldots\} \subseteq \mathbb{Z}[z]
\]
of all irreducible polynomials with coprime coefficients of the algebraic numbers in $S$, such that $P_1(z) = z$. We also set $\overline{\mathbb{Q}} \setminus S = \{\beta_1, \beta_2, \ldots\}$ and define for each $n \in \mathbb{N}$
\begin{equation}\label{xn}
x_n = \max_{1 \leq i \leq n}\max_{1 \leq k \leq n} \{H(P_1(\beta_i)\cdots P_k(\beta_i))\}.
\end{equation}

Our goal is to prove that there exist uncountably many sequences $(\omega_n)_{n\geq 1}$ and $(\delta_n)_{n\geq 1}$ of positive integers such that the function
\begin{equation}\label{function}
    h(z)=1+\sum_{n=1}^{\infty}p^{\delta_n}z^{\omega_n}\prod_{k=1}^nP_k(z)
\end{equation}
is a $p$-adic entire function with exceptional set $S_f = S$. We choose these sequences  so  that
\begin{equation}\label{omega}
    \omega_n \geq \omega_{n-1}+ \sum _{k=1}^{n}\deg (P_k(z))
\end{equation}
for all $n\geq 1$, and (see \eqref{xn})
\begin{equation}\label{delta}
\delta_n\geq \omega_n^{n}\cdot \max\{\delta_{n-1}, \lceil \log x_{n-1} \rceil\}
\end{equation}
for all $n\geq 1$. Therefore, we have that
\[
\left|z^{\omega_n}\prod_{k=1}^nP_k(z)\right|\leq \max\{|z|^{2\omega_n}, \ 1\}
\]
for all $z\in \C_p$. By \eqref{delta}, we obtain that
\[
\left|p^{\delta_n}z^{\omega_n}\prod_{k=1}^nP_k(z)\right|\to 0
\]
as $n\to \infty$. Thus, we have that \eqref{function} defines an entire function with integer coefficients. Furthermore, for each $n$, we can choose $\delta_n$ and $\omega_n$ in more than one way, so the number of functions as in \eqref{function} is uncountable. Then, it follows from Corollary \ref{c1} that uncountably many of these functions are transcendental. It remains to be proved that each of these functions has $S$ as its exceptional set.

Note that $h(\alpha)\in\Q(\alpha)\subseteq\QQ$ for all $\alpha\in S$. Indeed, if $\alpha\in S$, then there is a polynomial $P_N(z)\in \mathcal{P}_S$ such that $P_N(\alpha)=0$, so
\[
h(\alpha)=1+\sum_{n=1}^{N-1}p^{\delta_n}\alpha^{\omega_n}\prod_{k=1}^nP_k(\alpha)\in \Q(\alpha)\subseteq \QQ.
\]

Next, we prove that $h(\beta)\in \C_p\backslash \QQ$ for all $\beta\in \QQ\backslash S$. Assume $\beta = \beta_j$ has degree $r$ and, aiming at a contradiction, admit $h(\beta)$  also algebraic with degree $l$. For each $n \in \N$, we define $\gamma_n=\gamma_n(\beta)$ the algebraic number given by
\[
\gamma_n=1+\sum_{k=1}^{n}p^{\delta_k}\beta^{\omega_k}\prod_{j=1}^kP_j(\beta)\in \Q(\beta).
\]
In particular, $\deg(\gamma_n)\leq r$ for all $n\geq1$. Moreover, note that if it were $\gamma_n = \gamma_{n+1}$ we would get $$\prod_{j=1}^{n+1}P_j(\beta) = 0,$$ which can not happen, since $S$ is closed under algebraic conjugation and $\beta \not \in S$. Therefore, we have $\gamma_n \neq \gamma_{n+1}$ for all $n \in \N$ and, as an immediate consequence, we have $h(\beta) \neq \gamma_n$ for infinitely many $n \in \N$. In particular,  it follows directly from Lemma \ref{Bugeaud} that
\begin{equation}\label{bla}
    |h(\beta) - \gamma_n| \gg H(\gamma_n)^{-l}
\end{equation} for infinitely many $n \in \N$. Furthermore, we apply Lemma \ref{HeightPoly} and get the following upper bound for $H(\gamma_n)$:
\begin{gather}
\begin{aligned}\label{ble}
     H(\gamma_n) & \leq e^{O(n)}\cdot \left ( H(p^{\delta_1}\beta^{\omega_1}P_1(\beta))\cdots H\left(p^{\delta_n}\beta^{\omega_n}\prod_{j=1}^{n}P_j(\beta)\right)\right)^{r^n}\\
      & \leq e^{O(n)}\left( e^{O(n)} H(p^{\delta_1}\beta^{\omega_1})\cdots H(p^{\delta_n}\beta^{\omega_n})x_n^n \right )^{r^{n+2}}\\
      & \leq e^{O(n)}\cdot \left( e^{O(n)}\cdot p^{O(n\delta_n)}\cdot H(\beta)^{O(n\omega_n)}\cdot e^{O(n\omega_n)} \cdot x_n^n\right)^{r^{n+2}}\\
      & \leq e^{O(n)}\cdot (e^{n\delta_n}x_n^n)^{r^{n+2}}\\
      & \leq e^{O(n\delta_n r^{n+2})}x_n^{nr^{n+2}}.
 \end{aligned}
\end{gather} for all $n > j$ and $x_n$ is as in \eqref{xn}. Then, we conclude from \eqref{bla} and \eqref{ble} that 
\begin{equation}\label{bli}
|h(\beta) - \gamma_n| \gg e^{-O(l\cdot n\delta_n r^{n+2})}x_n^{-l\cdot nr^{n+2}}.
\end{equation}

On the other hand, for all $n$ sufficiently large, we have
\begin{gather}\label{blo}
    \begin{aligned}
        |h(\beta) - \gamma_n| & = \bigg|\sum_{k > n}p^{\delta_k}\beta^{\omega_k}\prod_{j=1}^kP_j(\beta)\bigg|\\
        &\leq \max_{k>n}\left \{\bigg| p^{\delta_k}\beta^{\omega_k}\prod_{j=1}^kP_j(\beta)\bigg|\right\}\\
        & \leq \max_{k> n}\{|p^{\delta_k}||\beta|^{2\omega_k}\}\\
        &= |p^{\delta_{n+1}}||\beta|^{2\omega_{n+1}}.
    \end{aligned}
\end{gather}

We conclude from equations \eqref{bli} and \eqref{blo} that 
\[|p^{\delta_{n+1}}||\beta|^{2\omega_{n+1}} \gg e^{-O(l\cdot n\delta_n r^{n+2})}x_n^{-l\cdot nr^{n+2}},\] and a little bit of algebra allows us to conclude that
\begin{equation}\label{blu}
    \frac{\delta_{n+1}}{n r^{n+2} (\delta_n + \ln x_n)} \ll \frac{-l}{\ln p}  - \frac{2\omega_{n+1} \ln |\beta|}{n r^{n+2} (\delta_n + \ln x_n) \ln p},
\end{equation}
which is a contradiction since by \eqref{delta} the left-hand side goes to infinity while the right-hand side tends to $-l/\ln p$. Hence, $h(\beta)$ is transcendental as we wished to prove. This finishes the proof of the case $\rho = \infty$ by taking $f(z) = h(z)$.

Now admit $\rho < \infty$. The final step of our proof is to create a $p$-adic analytic function $g(z)$ with rational coefficients and radius of convergence $\rho$ satisfying $g(\overline{\mathbb{Q}}\cap B(0,\rho)) \subseteq \overline{\mathbb{Q}}$. The procedure is somewhat similar to the one we just adopted when creating the functions $h(z)$. 

Let $(a_n)_{n \geq 1}$ be the sequence of integers such that $a_n$ is the smaller integer for which $a_n/n$ is bigger than $\ln \rho / \ln p$ for all $n \geq 1$. It is easy to see that
\[
\lim_{n \to \infty} \frac{a_n}{n} = \frac{\ln \rho}{\ln p}.
\]

Next, we consider the power series 
\begin{equation}\label{function2}
    g(z) = 1 + \sum_{n=1}^{\infty} p^{\psi_n}z^{\omega_n} \prod_{k=1}^{n}P_k(z) = 1 + \sum_{i=1}^{\infty}c_iz^i \in \mathbb{Q}[[z]].
\end{equation}
where the $\omega_n$ and $P_k(z) \in \Z[z]$ are chosen as in \eqref{function} with $S = \overline{\mathbb{Q}}\cap B(0,\rho)$. Moreover, in order to define the sequence $(\psi_n)_{n\geq1}$, we consider the expansion
\[
\prod_{k=1}^{n}P_k(z)=\sum_{i=1}^{N_n}d_i^{(n)}z^{i},
\]
where $N_n=\sum_{j=1}^{n}\deg(P_j(z))$ and define
\begin{equation}
    \psi_n = \max_{1 \leq i \leq N_n}\{ a_{\omega_n + i} - \nu_p(d_i^{(n)})\}.
\end{equation}
In particular, the coefficients $c_j$ of $g(z)$ are given by
\[c_j =  \begin{cases}p^{\psi_n}d_i^{(n)}, \ \ \text{for} \ \ j = \omega_n + i, \ \ i \in \{1,\ldots, N_n\}  \ \text{for some} \ n\in \N ,\\
   0, \ \ \text{for } \ \ \omega_n + N_n < j \leq \omega_{n+1}, \ \text{for some} \ n\in \N.
       
    \end{cases}.\]
    
We claim that $g(z)\in\Q_{\rho}[[z]]$ and $g(\QQ\cap B(0,\rho))\subseteq\QQ$, additionally if $\rho\geq 1$, then we have that $g(z)\in\Z_{\rho}[[z]]$. Indeed, we remember that the radius of convergence $\rho_g$ of $g(z)$ is given by
\[
\rho_g = \frac{1}{\limsup \sqrt[i]{|c_i|}}.
\] 
Therefore, by the choice of $\psi_n$, we have that $\psi_n + \nu_p(d_i^{(n)}) \geq a_{\omega_n + i}$ for all $i \in \{1,\ldots, N_n\}$. In particular
\[-\frac{\psi_n + \nu_p(d_i^{(n)})}{\omega_n + i} \leq - \frac{a_{\omega_n + i}}{\omega_n + i}\\
    \iff \sqrt[\omega_n + i]{|c_{\omega_n + i}|} \leq p^{-\frac{a_{\omega_n + i}}{\omega_n + i}}\] for all $i \in \{1,\ldots, N_n\}$. We observe that for each $n \in \N$ the upper bound above is achieved for some index $i$. By the choice of the $\omega_n$ (see \eqref{omega}) we have
    in each ``block'' of nonzero coefficients of the power series $g$ the maximum value of  $\sqrt[i]{|c_i|}$ being $p^{-\frac{a_i}{i}}$ for some $i$, we conclude by the choice of the sequence $(a_i)$ that $\limsup \sqrt[i]{|c_i|} = \frac{1}{\rho}$, which implies that the radius of convergence of $g$ is precisely $\rho$, as we wanted to show. Clearly, $g(\alpha) \in \overline{\mathbb{Q}}$ for all algebraic numbers in $B(0,\rho)$. Furthermore, if $\rho \geq 1$, we have $\ln\rho/\ln p \geq 0$ and the sequence $(a_n)$ must be taken with the $a_n$ being natural numbers. In particular, the coefficients $c_i$ in \eqref{function2} are integer numbers. 
    
    Finally, taking $f(z)=h(z)+g(z)$, we obtain that there exist uncountably many $p$-adic transcendental analytic functions $f\in\mathbb{Q}_{\rho}[[z]]$ such that $S_f=S$. Moreover, if $\rho\geq1$, then $f$ can be taken in $\Z_{\rho}[[z]]$ and this completes the proof.
\qed

\begin{remark}
    Note that the transcendence of the function $h(z)$ as well as that of the function $f(z)$ was obtained by Corollary \ref{c1}, which guarantees that there are only a countable number of algebraic functions with coefficients in $\mathbb{Q}$. However, we find it important to observe that all functions $h(z)$ which we constructed are transcendental. In fact, given an entire function $h(z) \in \mathbb{C}_p$, if $h(z)$ has only a finite number of zeros, then $h(z)$ is a polynomial function (see \cite[Chap. 6]{Robert}). Thus, it is not difficult to prove that a $p$-adic entire function is algebraic if and only if it is a polynomial function. Consequently, all the functions $h(z)$ are transcendental. 
\end{remark}
    
\section{Analytic Functions with Algebraic Coefficients}

The proof of Theorem \ref{theo2} will be presented as a consequence of the following more general theorem concerning power series with coefficients in dense subsets of $\mathbb{C}_p$ and a radius of convergence $\rho$.

\begin{theorem}\label{theo3}
Let $\rho \in (0, \infty]$ and let $\mathbb{K}$ be a dense subset of $\mathbb{C}_p$. Let $X$ be a countable subset of $B(0, \rho)$. For each $\alpha \in X$, fix a dense subset $E_{\alpha} \subseteq \mathbb{C}_p$ such that if $0 \in X$, then $E_0 \cap \mathbb{K} \neq \emptyset$. Then there exist uncountably many analytic functions $f \in \mathbb{K}_{\rho}[[z]]$ such that $f(\alpha) \in E_{\alpha}$ for all $\alpha \in X$.
\end{theorem}

\begin{proof}
Let $\rho \in (0, \infty]$ be a positive radius of convergence. Consider the function $g(z) \in \mathbb{Q}_{\rho}[[z]]$ as in the proof of Theorem \ref{theo2}. Our desired functions $f(z) \in \mathbb{K}[[z]]$ will be given as perturbations of the function $g$. In other words, our goal is to show that there exist uncountably many ways to construct a $p$-adic entire function $h(z) \in \mathbb{C}_p[[z]]$ such that the function $f(z) = g(z) + h(z) \in \mathbb{K}_{\rho}$ and $f(\alpha) \in E_{\alpha}$ for all $\alpha \in X$.

Write $X = \{\alpha_1, \alpha_2, \ldots\}$ and suppose, without loss of generality, that $X$ is infinite and $0 \in X$, say $\alpha_1 = 0$. Since $E_0 \cap \mathbb{K} \neq \emptyset$, there exists $\delta_0 \in \mathbb{C}_p$ such that $a_0 + \delta_0 \in E_0 \cap \mathbb{K}$. Setting $f_0(z) = g(z) + \delta_0$, we have that the first coefficient $a_0 + \delta_0$ of $f_0(z)$ belongs to $\mathbb{K}$ and $f_0(\alpha_1) \in E_{\alpha_1}$.

Note that, for any dense subset $D \subseteq \mathbb{C}_p$, if $a, b \in \mathbb{C}_p$ with $b \neq 0$, fixing any $r > 0$, we can choose infinitely many $d \in B(0, r) \setminus \{0\}$ such that $a + bd \in D$. Indeed, just choose $\alpha \in D \cap B(a, |b| \cdot r)$ and set $d = \frac{\alpha - a}{b}$. In particular, since $\mathbb{K}$ is a dense subset of $\mathbb{C}_p$, for any real number $\sigma_1 > 0$, there exist infinitely many $\delta_1 \in B(0, p^{\sigma_1}) \setminus \{0\}$ such that $a_1 + \delta_1 \in \mathbb{K}$. Choose one and set $f_1(z) = f_0(z) + \delta_1 z$. This arrangement ensures that both the first and second coefficients of $f_1(z)$ belong to $\mathbb{K}$. Similarly, since $\alpha_2 \neq 0$ and $E_{\alpha_2}$ is a dense subset of $\mathbb{C}_p$, there exist infinitely many $\epsilon_1 \in B(0, p^{\sigma_1}) \setminus \{0\}$ such that $f_1(\alpha_2) + \epsilon_1 \alpha_2^2 \in E_{\alpha_2}$. We choose one and set $\tilde{f}_1(z) = f_1(z) + \epsilon_1 z^2$. Therefore, both the first and second coefficients of $\tilde{f}_1(z)$ belong to $\mathbb{K}$ and $\tilde{f}_1(\alpha_i) \in E_{\alpha_i}$ for $i = 1, 2$.

Now, we proceed by induction. Suppose that
\[
\tilde{f}_{m}(z) = g(z) + \tilde{P}_m(z) = \sum_{k \geq 0} c_k^{(m)} z^k
\]
has been constructed such that
\[
\tilde{P}_m(z) = \delta_0 + \sum_{k=1}^{m} \left( \delta_k z^k + \epsilon_k z^{k+1} \right) \cdot \prod_{j=2}^{k} (z - \alpha_j) \in \mathbb{C}_p[z],
\]
with $\delta_k, \epsilon_k \in B(0, p^{\sigma_k}) \setminus \{0\}$, $c_k^{(m)} \in \mathbb{K}$ for all $k = 0, \ldots, m$, and $\tilde{f}_m(\alpha_{j}) \in E_{\alpha_{j}}$ for $j = 1, \ldots, m+1$.

Let us construct $\tilde{f}_{m+1}(z)$. Firstly, since $\mathbb{K}$ is a dense subset of $\mathbb{C}_p$ and $\alpha_i \neq 0$ for all $i \geq 2$, we can define
\[
f_{m+1}(z) = \tilde{f}_{m}(z) + \delta_{m+1} z^{m+1} \prod_{k=2}^{m+1} (z - \alpha_k),
\]
choosing $\delta_{m+1}$ in $B(0, p^{\sigma_{m+1}}) \setminus \{0\}$ such that
\[
c^{(m)}_{m+1} + (-1)^{m-1} \alpha_2 \cdots \alpha_{m+1} \cdot \delta_{m+1} \in \mathbb{K}.
\]
Next, since $E_{\alpha_{m+2}}$ is a dense subset of $\mathbb{C}_p$, we can set
\[
\tilde{f}_{m+1}(z) = f_{m+1}(z) + \epsilon_{m+1} z^{m+2} \prod_{k=2}^{m+1} (z - \alpha_k),
\]
where we choose $\epsilon_{m+1}$ in $B(0, p^{\sigma_{m+1}}) \setminus \{0\}$ such that
\[
\tilde{f}_{m+1}(\alpha_{m+2}) = f_{m+1}(\alpha_{m+2}) + \epsilon_{m+1} \alpha_{m+2}^{m+2} \prod_{k=2}^{m+1} (\alpha_{m+2} - \alpha_k) \in E_{\alpha_{m+2}}.
\]

Hence, by induction, we conclude that for every suitable sequence $(\sigma_n)_{n \geq 0}$ there exists a sequence $(\tilde{f}_n)_{n \geq 0}$ of power series 
\[
\tilde{f}_{n}(z) = g(z) + \tilde{P}_n(z) = \sum_{k \geq 0} c_k^{(n)} z^k
\]
such that
\[
\tilde{P}_n(z) = \delta_0 + \sum_{k=1}^{n} (\delta_k z^k + \epsilon_k z^{k+1}) \cdot \prod_{j=2}^{k} (z - \alpha_j) \in \mathbb{C}_p[z],
\]
with $\delta_k, \epsilon_k \in B(0, p^{\sigma_k}) \setminus \{0\}$, $c_k^{(n)} \in \mathbb{K}$ for all $k=0, \ldots, n$, and $\tilde{f}_n(\alpha_{j}) \in E_{\alpha_{j}}$ for $j=1, \ldots, n+1$.

Now we choose a suitable sequence $(\sigma_n)_{n \geq 0}$ such that $\tilde{P}_n(z) \to h(z)$ as $n \to \infty$ for any choice of $\delta_k, \epsilon_k \in B(0, p^{\sigma_k}) \setminus \{0\}$, where $h(z)$ is an entire function.

Observe that for each $n \in \mathbb{N}$ and $z \in \mathbb{C}_p$ we have
\[
\left| z^{n+1} \prod_{j=2}^{n} (z - \alpha_j) \right| \leq \max \{ |z|^{2n}, \ \rho^{2n} \}.
\]
Take $\sigma_n = -n^n$ for all $n \geq 0$ and note that $M^{2n} / p^{-n^n} \to 0$ as $n \to \infty$, for all fixed $M > 0$.

Hence, with $\sigma_n = -n^n$ for all $n \geq 0$, we have that $\tilde{P}_n(z) \to h(z)$, where $h(z)$ is an entire function, and $f(z) = g(z) + h(z) \in \mathbb{K}_{\rho}[[z]]$ is such that $f(\alpha_i) \in E_{\alpha_i}$ for all integers $i \geq 1$. Since, for all $k \geq 1$, there exist infinitely many possible choices for $\epsilon_k$ and $\delta_k$, we conclude that there exist uncountably many analytic functions $f \in \mathbb{K}_{\rho}[[z]]$ satisfying the conditions of Theorem \ref{theo3}. This concludes the proof.
\end{proof}

\subsection*{Proof that Theorem \ref{theo3} implies Theorem \ref{theo2}}
In the statement of Theorem \ref{theo3}, take $X = \mathbb{Q} \cap B(0, \rho)$ and $\mathbb{K} = \mathbb{Q}$. Write $S = \{\alpha_1, \alpha_2, \ldots\}$ and $\mathbb{Q} \setminus S = \{\beta_1, \beta_2, \ldots\}$ (one of them may be finite). Now, we fix
\[
E_{\alpha} := \begin{cases}
\mathbb{Q} & \text{if} \quad \alpha \in S, \\
\mathbb{C}_p \setminus \mathbb{Q}_p & \text{if} \quad \alpha \in \mathbb{Q} \setminus S.
\end{cases}
\]

By Theorem \ref{theo3}, there exist uncountably many analytic functions $f \in \mathbb{Q}_{\rho}[[z]]$ such that $f(\alpha) \in E_{\alpha}$ for all $\alpha \in \mathbb{Q}$, that is, $S_f = S$. Since there exist uncountably many analytic functions $f \in \mathbb{Q}_{\rho}[[z]]$ such that $S_f = S$ and $\mathbb{Q}$ is a countable subfield of $\mathbb{C}_p$, we conclude that there exist uncountably many $p$-adic transcendental analytic functions $f \in \overline{\mathbb{Q}}_{\rho}[[z]]$ such that $S_f = S$.
\qed

\begin{remark}
Finally, we note that our construction relies entirely on the countability of $\overline{\mathbb{Q}}$. Therefore, we do not know if for uncountable subsets $S$ of $\overline{\mathbb{Q}}_p \cap B(0, \rho)$ there exist transcendental analytic functions $f$ with radius of convergence $\rho \in (0, +\infty]$ such that for $\alpha \in \overline{\mathbb{Q}}_p$, we have $f(\alpha) \in \overline{\mathbb{Q}}_p$ if and only if $\alpha \in \mathcal{S}$, which would be another generalization of Mahler's Problem $C$ over $\mathbb{C}_p$. We consider this to be an interesting question.
\end{remark}

	\begin{comment}
	\bibliographystyle{amsplain}
	\renewcommand{\bibname}{ Referências}
	
	\bibliography{refs}
	\end{comment}
\end{document}